%% file: MTNS2020_Kolar.tex
\documentclass{ifacconf}

\usepackage{graphicx}      
\usepackage{natbib}        

\usepackage{xcolor}
\usepackage{amsmath}
\usepackage{amssymb}
\usepackage{enumitem}
\usepackage{comment}
\newenvironment{proof}{\emph{Proof.}}{\hfill~$\square$ \\[2mm]}

\begin{document}
\begin{frontmatter}

\title{Linearized Controllability Analysis of Semilinear Partial Differential Equations\thanksref{footnoteinfo}} 

\thanks[footnoteinfo]{This work has been supported by the Austrian Science Fund (FWF) under grant number P 29964-N32.}

\author[First]{Bernd Kolar} 
\author[First]{Markus Sch{\"o}berl} 

\address[First]{Institute of Automatic Control and Control Systems Technology, Johannes Kepler University Linz, Altenbergerstrasse 66, 4040 Linz, Austria (e-mail: bernd.kolar@jku.at, markus.schoeberl@jku.at).}

\begin{abstract}                
It is well-known that the controllability of finite-dimensional nonlinear systems can be established by showing the controllability of the linearized system. However, this classical result does not generalize to infinite-dimensional nonlinear systems. In this paper, we restrict ourselves to semilinear infinite-dimensional systems, and show that the exact controllability of the linearized system implies exact controllability of the nonlinear system. The restrictions concerning the nonlinear operator are similar to those that can be found in the literature about the linearized stability analysis of semilinear systems.
\end{abstract}


\end{frontmatter}

\input{MTNS2020_sub}

\begin{ack}
The first author would like to thank Hans Zwart for interesting discussions about semilinear systems during a research visit at the University of Twente in March 2019.
\end{ack}

\bibliography{/home/bernd/Dokumente/Publikationen/Bibtex_Literaturdatenbank/Bibliography_Bernd}             
                                                   







\end{document}

%% file: MTNS2020_sub.tex
\section{Introduction}

Studying system-theoretic properties like controllability or observability
for nonlinear infinite-dimensional systems is in general a very difficult
task. Roughly speaking, systems theory for partial differential equations
can be divided into formal (algebraic or geometric) methods like e.g.
in \cite{Pommaret:1994} and \cite{Schoberl:2014} that are based
on the structure of the equations, and functional-analytic methods
that are rather based on the solutions. Whereas formal methods have
been proven to be very successful for finite-dimensional nonlinear
systems, in the infinite-dimensional case they suffer from the drawback
that the function spaces for e.g. the state and the input cannot be
properly specified. Thus, formal methods seem to be rather suited
for proving negative results like non-controllability or non-observability,
that can possibly be shown directly from the structure of the equations,
see e.g. \cite{KolarRamsSchoberl:2018} or \cite{KolarSchoberl:2019}.
For proving positive results, in contrast, a functional-analytic approach
seems to be indispensable.

For finite-dimensional nonlinear systems, it is well-known that stability,
controllability, or observability can be established in a straightforward
way by proving the corresponding property for the linearized system,
see e.g. \cite{NijmeijervanderSchaft:1990} or \cite{Khalil:2002}.
Unfortunately, these classical results do not generalize to infinite-dimensional
nonlinear systems. So far, the existing literature deals mainly with
the linearized stability analysis of infinite-dimensional nonlinear
systems, see e.g. \cite{DeschSchappacher:1986}, \cite{Smoller:1994},
\cite{Kato:1995}, \cite{JamalChowMorris:2014}, or \cite{JamalMorris:2018}.
Particularly, it is shown in \cite{Smoller:1994} that for semilinear
systems with a nonlinear operator that is subject to certain restrictions
the stability of the linearized system implies -- like in the finite-dimensional
case -- the (local) stability of the original system. In the present
paper, we pursue a similar approach with respect to the exact controllability
problem for infinite-dimensional semilinear systems with distributed
input. Exact controllability means that the controllability map of
the system is surjective, and the basic idea of our approach is to
apply the local surjectivity theorem to this controllability map in
order to establish a connection between the exact controllability
of the linearized system and the (local) exact controllability of
the nonlinear system. %
{} Even though the controllability problem is quite different from the
stability problem, we need conditions on the nonlinear operator of
the semilinear system that are very similar to those in \cite{Smoller:1994}.
It should also be noted that in contrast to the stability analysis
we are dealing here with non-autonomous systems, and, as mentioned
in \cite{SchmidDashkovskiyJacobLaasri:2019}, there exist only very
few papers on non-autonomous semilinear systems in the context of
control theory.

\section{Preliminaries}

Throughout the paper, we need the concept of a Fr\'{e}chet derivative
of maps between infinite-dimensional spaces.
\begin{defn}
(Fr\'{e}chet Derivative) A map $f:X\rightarrow Y$ from a Banach
space $X$ to a Banach space $Y$ is Fr\'{e}chet differentiable at
$x\in X$, if there exists a bounded linear operator $Df(x):X\rightarrow Y$
such that
\begin{equation}
\lim_{\left\Vert h\right\Vert _{X}\rightarrow0}\frac{\left\Vert f(x+h)-f(x)-Df(x)h\right\Vert _{Y}}{\left\Vert h\right\Vert _{X}}=0\,.\label{eq:Condition_Frechet}
\end{equation}
The map is Fr\'{e}chet differentiable if it is Fr\'{e}chet differentiable
at every $x\in X$, and it is continuously Fr\'{e}chet differentiable
if the Fr\'{e}chet derivative $Df(x)$ depends continuously on $x$.
\end{defn}
For bounded linear operators between Banach spaces, we use the usual
operator norm
\[
\left\Vert T\right\Vert =\sup_{x\in D(T),\,x\neq0}\frac{\left\Vert Tx\right\Vert _{Y}}{\left\Vert x\right\Vert _{X}}\,,
\]
and denote it by $\left\Vert \cdot\right\Vert $ without a subscript.

Our approach to the linearized controllability analysis is based on
the local surjectivity theorem, that can be found e.g. in \cite{AbrahamMarsdenRatiu:1988}.
\begin{thm}
(Local Surjectivity Theorem) Let $X$ and $Y$ be Banach spaces and
$V\subset X$ be open. If the map $f:V\subset X\rightarrow Y$ is
continuously Fr\'{e}chet differentiable and $Df(x_{0})$ is surjective
for some $x_{0}\in V$, then $f$ is locally surjective. That is,
there exist open neighborhoods $V_{1}$ of $x_{0}$ and $W_{1}$ of
$f(x_{0})$ such that $\left.f\right|_{V_{1}}:V_{1}\rightarrow W_{1}$
is surjective.
\end{thm}
We also make frequent use of Gronwall's lemma in the form presented
in \cite{Zeidler:1986}.
\begin{lem}
(Gronwall) Let $f,g:[t_{0},\tau]\rightarrow\mathbb{R}$ be continuous
functions, with $g$ nondecreasing, and which, for fixed $K>0$, satisfy
the inequality
\[
f(t)\leq g(t)+K\intop_{t_{0}}^{t}f(s)ds\,,\quad\forall t\in[t_{0},\tau]\,.
\]
Then
\[
f(t)\leq g(t)e^{K(t-t_{0})}\,,\quad\forall t\in[t_{0},\tau]\,.
\]
\end{lem}

\section{Semilinear Systems}

We consider semilinear systems of the form
\begin{equation}
\dot{x}(t)=Ax(t)+f(x(t))+Bu(t)\,,\quad x(t_{0})=x_{0}\,,\label{eq:Sys_Ax_f}
\end{equation}
where $A$ is the infinitesimal generator of a strongly continuous
semigroup $T(t)$ on a Hilbert space $X$, the linear operator $B:U\rightarrow X$
is bounded, and the nonlinear map $f:X\rightarrow X$ is continuously
Fr\'{e}chet differentiable. %
Furthermore, we assume that $f(0)=0$ and $Df(0)=0$. With these assumptions,
\[
(x_{s},u_{s})=(0,0)
\]
is an equilibrium of the system (\ref{eq:Sys_Ax_f}), and the linearization
about this equilibrium is given by
\begin{equation}
\Delta\dot{x}(t)=A\Delta x(t)+B\Delta u(t)\,,\quad\Delta x(t_{0})=\Delta x_{0}\,,\label{eq:Linearized_system}
\end{equation}
which is just the linear part of the system.
\begin{rem}
It should be noted that the linearization of the system is based on
the G\^{a}teaux derivative, see e.g. \cite{JamalChowMorris:2014}
or \cite{JamalMorris:2018}. Since the operator $A$ is typically
unbounded, the right-hand side of (\ref{eq:Sys_Ax_f}) does not possess
a Fr\'{e}chet derivative with respect to $x$.
\end{rem}
As discussed in \cite{Pazy:1983}, classical solutions of the semilinear
system (\ref{eq:Sys_Ax_f}) satisfy the integral equation
\begin{equation}
x(t)=T(t)x(t_{0})+\intop_{t_{0}}^{t}T(t-s)\left(f(x(s))+Bu(s)\right)ds\,.\label{eq:integral_equation}
\end{equation}
Therefore, continuous solutions of the integral equation (\ref{eq:integral_equation})
are called mild solutions of (\ref{eq:Sys_Ax_f}), see also \cite{CurtainZwart:1995}
for the linear case. For the controllability problem with $t_{0}=0$
and initial condition $x(0)=0$, we have
\begin{equation}
x(t)=\intop_{0}^{t}T(t-s)\left(f(x(s))+Bu(s)\right)ds\,.\label{eq:controllability problem}
\end{equation}

Throughout the paper, we assume that for some $\tau>0$ and all inputs
$u(t)$ in an open neighborhood
\[
\mathcal{U}\subset L_{2}([0,\tau];U)
\]
of $u(t)=0$ the semilinear system (\ref{eq:Sys_Ax_f}) with $x(0)=0$
has a unique mild solution on the interval $[0,\tau]$, i.e., a unique
continuous solution of (\ref{eq:controllability problem}).
\begin{rem}
Proving the existence and uniqueness of solutions would require a
(repeated) application of Banach's fixed-point theorem.
\end{rem}
In the following, we denote the solution for an input $u(t)\in\mathcal{U}$
by
\[
x(t)=S_{t}(u)\,,\quad t\in[0,\tau]\,.
\]
The map
\begin{equation}
S_{\tau}(u):\mathcal{U}\rightarrow X\label{eq:S_tau}
\end{equation}
with $t=\tau$ is the controllability map of the semilinear system
on $[0,\tau]$, and we call the system locally exactly controllable
on $[0,\tau]$ if this map is locally surjective. In other words,
for every final state $x(\tau)\in X$ in an open neighborhood of the
origin there must exist an input trajectory $u(t)$ on the time interval
$[0,\tau]$ such that
\[
x(\tau)=S_{\tau}(u)\,.
\]
The basic idea is now to use the local surjectivity theorem in order
to establish a connection between the local exact controllability
of the semilinear system (\ref{eq:Sys_Ax_f}) and the exact controllability
of the linearized system (\ref{eq:Linearized_system}).%

\begin{thm}
\label{thm:ExactControllability}Assume that the controllability map
(\ref{eq:S_tau}) of the semilinear system (\ref{eq:Sys_Ax_f}) satisfies
the following conditions:
\begin{enumerate}
\item[(A1)] \label{enu:condition_contFrechet}$S_{\tau}(u)$ is continuously
Fr\'{e}chet differentiable with respect to $u$ in an open neighborhood
of $u=0$.
\item[(A2)] \label{enu:condition_Frechet_lin}The Fr\'{e}chet derivative $DS_{\tau}(0)$
at $u=0$ coincides with the controllability map
\[
\intop_{0}^{\tau}T(\tau-s)B\Delta u(s)ds:L_{2}([0,\tau];U)\rightarrow X
\]
of the linearized system (\ref{eq:Linearized_system}) on $[0,\tau]$.
\end{enumerate}
Then exact controllability of the linearized system (\ref{eq:Linearized_system})
on $[0,\tau]$ implies local exact controllability of the semilinear
system (\ref{eq:Sys_Ax_f}) on $[0,\tau]$.
\end{thm}
\begin{proof}
The condition (A1) is a prerequisite for the application of the local
surjectivity theorem. With condition (A2) and the local surjectivity
theorem, surjectivity of the controllability map of the linearized
system implies local surjectivity of the controllability map of the
semilinear system. Consequently, exact controllability of the linearized
system implies local exact controllability of the semilinear system.
\end{proof}
The practical use of this theorem is of course limited, since it would
require knowledge of the controllability map (\ref{eq:S_tau}) of
the semilinear system. Thus, in the remainder of the paper, we translate
the conditions (A1) and (A2) of Theorem \ref{thm:ExactControllability}
into (sufficient) conditions on the nonlinear term $f$ of the system
(\ref{eq:Sys_Ax_f}). The main difficulty consists in proving the
continuous Fr\'{e}chet differentiability of $S_{\tau}(u)$.

\section{Fr\'{e}chet Differentiability of the Controllability Map}

In this section, we show that the conditions (A1) and (A2) of Theorem
\ref{thm:ExactControllability} are satisfied for all semilinear systems
(\ref{eq:Sys_Ax_f}) where the nonlinear operator $f$ satisfies the
following additional assumptions.
\begin{enumerate}
\item[(B1)] \label{enu:assumption_square}For every bounded set $\mathcal{B}\subset X$,
there exist positive constants $\alpha,\gamma\in\mathbb{R}$ such
that
\[
\left\Vert f(x_{1})\hspace{-0.5mm}-\hspace{-0.5mm}f(x_{2})\hspace{-0.5mm}-\hspace{-0.5mm}Df(x_{2})(x_{1}\hspace{-0.5mm}-\hspace{-0.5mm}x_{2})\right\Vert _{X}\leq\alpha\left\Vert x_{1}\hspace{-0.5mm}-\hspace{-0.5mm}x_{2}\right\Vert _{X}^{1+\gamma}
\]
for all $x_{1},x_{2}\in\mathcal{B}$.
\item[(B2)] \label{enu:assumption_Frechet}The Fr\'{e}chet derivative $Df$
is locally Lipschitz continuous.
\end{enumerate}
The assumption (B1) with $\gamma=1$ is also used in \cite{Smoller:1994}
for the linearized stability analysis of semilinear autonomous systems.%

We proceed in several steps. First, we show that for all $t\in[0,\tau]$
the map $S_{t}(u):\mathcal{U}\rightarrow X$ is locally Lipschitz
continuous (Lemma \ref{lem:SolutionMapContinuous}). Based on this
result, we prove that $S_{t}(u)$ is also Fr\'{e}chet differentiable
with respect to $u$, and that the Fr\'{e}chet derivative coincides
with the solution operator (controllability map) of the linearized
system (Theorem \ref{thm:FrechetDerivative}). Finally, we prove that
$S_{t}(u)$ is even continuously Fr\'{e}chet differentiable (Theorem
\ref{thm:ContinuousFrechet}).%

\begin{lem}
\label{lem:SolutionMapContinuous}There exists a constant $c$ such
that
\[
\left\Vert S_{t}(u_{1})-S_{t}(u_{2})\right\Vert _{X}\leq c\left\Vert u_{1}-u_{2}\right\Vert _{L_{2}}
\]
for all $t\in[0,\tau]$ and $u_{1},u_{2}\in\mathcal{U}$.
\end{lem}
\begin{proof}
Let $x_{1}(t)=S_{t}(u_{1})$ and $x_{2}(t)=S_{t}(u_{2})$ denote the
solutions for two input trajectories $u_{1},u_{2}\in\mathcal{U}$.
These solutions satisfy the integral equations
\[
x_{1}(t)=\intop_{0}^{t}T(t-s)\left(f(x_{1}(s))+Bu_{1}(s)\right)ds
\]
and
\[
x_{2}(t)=\intop_{0}^{t}T(t-s)\left(f(x_{2}(s))+Bu_{2}(s)\right)ds\,.
\]
For the norm of the difference $x_{1}(t)-x_{2}(t)$, we get the estimate%
\begin{align}
\left\Vert x_{1}(t)-x_{2}(t)\right\Vert _{X} & \leq\left\Vert \intop_{0}^{t}T(t-s)\left(f(x_{1}(s))\hspace{-1mm}-\hspace{-1mm}f(x_{2}(s))\right)\hspace{-1mm}ds\right\Vert _{X}\nonumber \\
 & \phantom{\leq}+\left\Vert \intop_{0}^{t}T(t-s)B(u_{1}(s)-u_{2}(s))ds\right\Vert _{X}\nonumber \\
 & \leq\hspace{-1mm}\intop_{0}^{t}\hspace{-1mm}\left\Vert T(t-s)\right\Vert \left\Vert f(x_{1}(s))\hspace{-1mm}-\hspace{-1mm}f(x_{2}(s))\right\Vert _{X}ds\nonumber \\
 & \phantom{\leq}+k\left\Vert u_{1}-u_{2}\right\Vert _{L_{2}}\nonumber \\
 & \leq ML\intop_{0}^{t}\left\Vert x_{1}(s)-x_{2}(s)\right\Vert _{X}ds\nonumber \\
 & \phantom{\leq}+k\left\Vert u_{1}-u_{2}\right\Vert _{L_{2}}\label{eq:Stu_Continuity_Inequality}
\end{align}
with $M=\sup_{t\in[0,\tau]}\left\Vert T(t)\right\Vert $ and some
$k>0$. Since the continuous Fr\'{e}chet differentiability of $f$
guarantees only local Lipschitz continuity, the Lipschitz constant
$L$ depends of course on
\[
\sup_{u\in\mathcal{U},t\in[0,\tau]}\left\Vert S_{t}(u)\right\Vert _{X}\,,
\]
i.e., on the maximal ``size'' of the solutions with input trajectories
$u\in\mathcal{U}$. We have also used the fact that the term
\[
\intop_{0}^{t}T(t-s)Bu(s)ds
\]
is just the controllability map of the linear part of the system and
therefore bounded with some constant $k$, see e.g. \cite{CurtainZwart:1995}.
Applying Gronwall's lemma %
to (\ref{eq:Stu_Continuity_Inequality}) yields
\[
\left\Vert x_{1}(t)-x_{2}(t)\right\Vert _{X}\leq k\left\Vert u_{1}-u_{2}\right\Vert _{L_{2}}e^{MLt}
\]
and finally
\[
\left\Vert x_{1}(t)-x_{2}(t)\right\Vert _{X}\leq\underbrace{ke^{ML\tau}}_{c}\left\Vert u_{1}-u_{2}\right\Vert _{L_{2}}\,,\quad\forall t\in[0,\tau]\,,
\]
which completes the proof.
\end{proof}
With help of Lemma \ref{lem:SolutionMapContinuous}, we can now show
that the solutions $x(t)=S_{t}(u)$ are Fr\'{e}chet differentiable
with respect to $u$. The proof is divided in fact into two parts.
First, we show that if the Fr\'{e}chet derivative exists it coincides
with the solution operator of the linearized system. Subsequently,
we prove that the solution operator of the linearized system satisfies
the condition (\ref{eq:Condition_Frechet}) for a Fr\'{e}chet derivative.
\begin{thm}
\label{thm:FrechetDerivative}For all $t\in[0,\tau]$, the map
\[
S_{t}(u):L_{2}([0,\tau];U)\rightarrow X
\]
is Fr\'{e}chet differentiable with respect to $u$ at every $\bar{u}\in\mathcal{U}$.
The Fr\'{e}chet derivative $DS_{t}(\bar{u})$ coincides with the
solution operator (controllability map)
\[
L_{t}\Delta u:L_{2}([0,\tau];U)\rightarrow X
\]
of the linearized, time-variant system\footnote{The system (\ref{eq:Linearized-time-variant-system}) is the linearization
of (\ref{eq:Sys_Ax_f}) along the trajectory $x(t)=S_{t}(\bar{u})$.
For $\bar{u}=0$ we have $S_{t}(\bar{u})=0$, and (\ref{eq:Linearized-time-variant-system})
becomes the linear, time-invariant system (\ref{eq:Linearized_system}).
Thus, the condition (A2) of Theorem \ref{thm:ExactControllability}
is contained as a special case.}
\begin{equation}
\Delta\dot{x}(t)=\left(A+Df(S_{t}(\bar{u}))\right)\Delta x(t)+B\Delta u(t)\label{eq:Linearized-time-variant-system}
\end{equation}
with $\Delta x(0)=0$.
\end{thm}

\begin{proof}
If $S_{t}(u)$ is Fr\'{e}chet differentiable at $\bar{u}\in\mathcal{U}$
for all $t\in[0,\tau]$, then we can differentiate the integral equation
(\ref{eq:controllability problem}) and observe that the Fr\'{e}chet
derivative $DS_{t}(\bar{u})$ must satisfy the integral equation\footnote{It should be noted that the integrals in this equation are in general
not Lebesgue integrals but Pettis integrals, see also \cite{CurtainZwart:1995}.}
\begin{multline*}
DS_{t}(\bar{u})=\intop_{0}^{t}T(t-s)Df(S_{s}(\bar{u}))DS_{s}(\bar{u})ds\\
+\intop_{0}^{t}T(t-s)Bds\,.
\end{multline*}
The solution of the linearized system (\ref{eq:Linearized-time-variant-system})
meets
\begin{multline}
L_{t}\Delta u=\intop_{0}^{t}T(t-s)Df(S_{s}(\bar{u}))L_{s}\Delta uds\\
+\intop_{0}^{t}T(t-s)B\Delta u(s)ds\,,\label{eq:IntegralEq_LinearizedSys}
\end{multline}
and because of $D(L_{t}\Delta u)=L_{t}$ the solution operator $L_{t}$
satisfies exactly the same integral equation
\begin{multline}
L_{t}=\intop_{0}^{t}T(t-s)Df(S_{s}(\bar{u}))L_{s}ds+\intop_{0}^{t}T(t-s)Bds\,.\label{eq:IntegralEq_SolutionOp_LinearizedSys}
\end{multline}
Thus, if the Fr\'{e}chet derivative $DS_{t}(\bar{u})$ exists, it
coincides with the operator $L_{t}$.\footnote{Provided that the integral equations have a unique solution, but it
is well-known that the Fr\'{e}chet derivative is unique if it exists.} Consequently, we must prove that $L_{t}$ satisfies the condition
\begin{equation}
\lim_{\left\Vert \Delta u\right\Vert _{L_{2}}\rightarrow0}\frac{\left\Vert S_{t}(\bar{u}+\Delta u)-S_{t}(\bar{u})-L_{t}\Delta u\right\Vert _{X}}{\left\Vert \Delta u\right\Vert _{L_{2}}}=0\label{eq:Condition_Frechet_Lt}
\end{equation}
for a Fr\'{e}chet derivative of $S_{t}(u)$ at $\bar{u}$, cf. (\ref{eq:Condition_Frechet}).
For this purpose, we introduce the abbreviation
\[
\sigma(t)=S_{t}(\bar{u}+\Delta u)-S_{t}(\bar{u})-L_{t}\Delta u
\]
for the sum in the numerator of (\ref{eq:Condition_Frechet_Lt}).
Substituting
\begin{align*}
S_{t}(\bar{u}+\Delta u) & =\intop_{0}^{t}T(t-s)f(S_{s}(\bar{u}+\Delta u))ds\\
 & \phantom{=}+\intop_{0}^{t}T(t-s)B(\bar{u}(s)+\Delta u(s))ds\\
S_{t}(\bar{u}) & =\intop_{0}^{t}T(t-s)f(S_{s}(\bar{u}))ds\\
 & \phantom{=}+\intop_{0}^{t}T(t-s)B\bar{u}(s)ds\\
L_{t}\Delta u & =\intop_{0}^{t}T(t-s)Df(S_{s}(\bar{u}))L_{s}\Delta uds\\
 & \phantom{=}+\intop_{0}^{t}T(t-s)B\Delta u(s)ds
\end{align*}
according to the integral equations (\ref{eq:controllability problem})
and (\ref{eq:IntegralEq_LinearizedSys}) yields
\begin{multline}
\sigma(t)=\intop_{0}^{t}T(t-s)\left(f(S_{s}(\bar{u}+\Delta u))-f(S_{s}(\bar{u}))\right.\\
\left.-Df(S_{s}(\bar{u}))L_{s}\Delta u\right)ds\,.\label{eq:integral_eq_sigma}
\end{multline}
Now we write $f(S_{s}(\bar{u}+\Delta u))$ as
\begin{multline}
f(S_{s}(\bar{u}+\Delta u))=\\
f(S_{s}(\bar{u}))+Df(S_{s}(\bar{u}))\left(S_{s}(\bar{u}+\Delta u)-S_{s}(\bar{u})\right)+r(s)\label{eq:Expansion_f}
\end{multline}
with some rest $r(s)$,%
{} and substituting
\begin{multline*}
f(S_{s}(\bar{u}+\Delta u))-f(S_{s}(\bar{u}))=\\
Df(S_{s}(\bar{u}))\left(S_{s}(\bar{u}+\Delta u)-S_{s}(\bar{u})\right)+r(s)
\end{multline*}
into (\ref{eq:integral_eq_sigma}) yields
\[
\sigma(t)=\intop_{0}^{t}T(t-s)Df(S_{s}(\bar{u}))\sigma(s)ds+\intop_{0}^{t}T(t-s)r(s)ds\,.
\]
For the norm of $\sigma(t)$ we get the estimate
\begin{align*}
\left\Vert \sigma(t)\right\Vert _{X} & \leq\intop_{0}^{t}\left\Vert T(t-s)\right\Vert \left\Vert Df(S_{s}(\bar{u}))\right\Vert \left\Vert \sigma(s)\right\Vert _{X}ds+\\
 & \phantom{\leq}+\intop_{0}^{t}\left\Vert T(t-s)\right\Vert \left\Vert r(s)\right\Vert _{X}ds\\
 & \leq M\bar{c}\intop_{0}^{t}\left\Vert \sigma(s)\right\Vert _{X}ds+M\intop_{0}^{t}\left\Vert r(s)\right\Vert _{X}ds
\end{align*}
with
\[
\bar{c}=\sup_{t\in[0,\tau]}\left\Vert Df(S_{t}(\bar{u}))\right\Vert \,,
\]
and applying Gronwall's lemma yields
\begin{equation}
\left\Vert \sigma(t)\right\Vert _{X}\leq\left(M\intop_{0}^{t}\left\Vert r(s)\right\Vert _{X}ds\right)e^{M\bar{c}t}\,.\label{eq:Gronwall_sigma}
\end{equation}
Now we make use of the additional assumption (B1) to obtain an inequality
for $\left\Vert r(s)\right\Vert _{X}$. Applying (B1) to the right-hand
side of
\begin{multline*}
r(s)=f(S_{s}(\bar{u}+\Delta u))-f(S_{s}(\bar{u}))\\
-Df(S_{s}(\bar{u}))\left(S_{s}(\bar{u}+\Delta u)-S_{s}(\bar{u})\right)
\end{multline*}
yields
\[
\left\Vert r(s)\right\Vert _{X}\leq\alpha\left\Vert S_{s}(\bar{u}+\Delta u)-S_{s}(\bar{u})\right\Vert _{X}^{1+\gamma}\,,
\]
and with Lemma \ref{lem:SolutionMapContinuous} we get
\[
\left\Vert r(s)\right\Vert _{X}\leq\alpha c^{1+\gamma}\left\Vert \Delta u\right\Vert _{L_{2}}^{1+\gamma}\,,\quad\forall s\in[0,\tau]\,.
\]
Substituting this inequality into (\ref{eq:Gronwall_sigma}) results
in
\[
\left\Vert \sigma(t)\right\Vert _{X}\leq M\alpha c^{1+\gamma}te^{M\bar{c}t}\left\Vert \Delta u\right\Vert _{L_{2}}^{1+\gamma}\,,\quad t\in[0,\tau]\,,
\]
and because of
\[
\lim_{\left\Vert \Delta u\right\Vert _{L_{2}}\rightarrow0}\frac{\left\Vert \sigma(t)\right\Vert _{X}}{\left\Vert \Delta u\right\Vert _{L_{2}}}\leq\lim_{\left\Vert \Delta u\right\Vert _{L_{2}}\rightarrow0}M\alpha c^{1+\gamma}te^{M\bar{c}t}\left\Vert \Delta u\right\Vert _{L_{2}}^{\gamma}=0
\]
the condition (\ref{eq:Condition_Frechet_Lt}) for a Fr\'{e}chet
derivative is indeed satisfied.
\end{proof}
However, for the application of the local surjectivity theorem, it
is not enough to prove the Fr\'{e}chet differentiability. We have
to show that $S_{\tau}(u)$ is continuously Fr\'{e}chet differentiable,
i.e., that the Fr\'{e}chet derivative $DS_{\tau}(\bar{u})$ depends
continuously on the point $\bar{u}\in\mathcal{U}$.%

\begin{thm}
\label{thm:ContinuousFrechet}For all $t\in[0,\tau]$, the map
\[
S_{t}(u):L_{2}([0,\tau];U)\rightarrow X
\]
is continuously Fr\'{e}chet differentiable on $\mathcal{U}$.
\end{thm}
\begin{proof}
We have to show that the map $u\rightarrow DS_{t}(u)$ is continuous.
If $DS_{t}(u_{1})$ and $DS_{t}(u_{2})$ are the Fr\'{e}chet derivatives
of $S_{t}(u)$ at $u_{1},u_{2}\in\mathcal{U}$, then they satisfy
the integral equations
\[
DS_{t}(u_{1})=\intop_{0}^{t}T(t-s)\left(Df(S_{s}(u_{1}))DS_{s}(u_{1})+B\right)ds
\]
and
\[
DS_{t}(u_{2})=\intop_{0}^{t}T(t-s)\left(Df(S_{s}(u_{2}))DS_{s}(u_{2})+B\right)ds\,.
\]
The difference of these equations can be written as%
\[
\begin{array}{l}
DS_{t}(u_{1})-DS_{t}(u_{2})=\\
\intop_{0}^{t}\hspace{-0.5mm}\hspace{-0.5mm}T(t\hspace{-0.5mm}-\hspace{-0.5mm}s)\left(Df(S_{s}(u_{1}))DS_{s}(u_{1})\right.\hspace{-0.5mm}\hspace{-0.5mm}-\hspace{-0.5mm}\hspace{-0.5mm}\left.Df(S_{s}(u_{2}))DS_{s}(u_{2})\right)ds\\
=\intop_{0}^{t}T(t-s)\left(Df(S_{s}(u_{1}))\left(DS_{s}(u_{1})\right.\right.\\
\phantom{=}-\hspace{-0.5mm}\left.DS_{s}(u_{2})\right)\hspace{-0.5mm}+\hspace{-0.5mm}\left(Df(S_{s}(u_{1}))\right.\hspace{-0.5mm}\hspace{-0.5mm}-\hspace{-0.5mm}\hspace{-0.5mm}\left.\left.Df(S_{s}(u_{2}))\right)DS_{s}(u_{2})\right)ds\,.
\end{array}
\]
For the norm of the difference, we get the estimate
\begin{equation}
\begin{array}{l}
\left\Vert DS_{t}(u_{1})-DS_{t}(u_{2})\right\Vert \leq\\
\intop_{0}^{t}\left\Vert T(t-s)\right\Vert \left\Vert Df(S_{s}(u_{1}))\right\Vert \left\Vert DS_{s}(u_{1})-DS_{s}(u_{2})\right\Vert ds\\
+\intop_{0}^{t}\hspace{-0.5mm}\left\Vert T(t-s)\right\Vert \hspace{-0.5mm}\left\Vert Df(S_{s}(u_{1}))\hspace{-0.5mm}-\hspace{-0.5mm}Df(S_{s}(u_{2}))\right\Vert \hspace{-0.5mm}\left\Vert DS_{s}(u_{2})\right\Vert ds\\
\leq Mc_{1}\intop_{0}^{t}\left\Vert DS_{s}(u_{1})-DS_{s}(u_{2})\right\Vert ds\\
\phantom{\leq}+Mc_{2}\intop_{0}^{t}\left\Vert Df(S_{s}(u_{1}))-Df(S_{s}(u_{2}))\right\Vert ds
\end{array}\label{eq:S_continuous_frechet_inequality}
\end{equation}
with
\[
c_{1}=\underset{u\in\mathcal{U},t\in[0,\tau]}{\sup}\left\Vert Df(S_{t}(u))\right\Vert 
\]
and
\[
c_{2}=\underset{u\in\mathcal{U},t\in[0,\tau]}{\sup}\left\Vert DS_{t}(u)\right\Vert \,.
\]
With the Lipschitz continuity
\[
\left\Vert Df(S_{s}(u_{1}))-Df(S_{s}(u_{2}))\right\Vert \leq c_{3}\left\Vert S_{s}(u_{1})-S_{s}(u_{2})\right\Vert _{X}
\]
of $Df$ according to assumption (B2) and Lemma \ref{lem:SolutionMapContinuous},
we also get
\[
\left\Vert Df(S_{s}(u_{1}))-Df(S_{s}(u_{2}))\right\Vert \leq\underbrace{c_{3}c}_{c_{4}}\left\Vert u_{1}-u_{2}\right\Vert _{L_{2}}
\]
$\forall s\in[0,t]$. Thus, (\ref{eq:S_continuous_frechet_inequality})
can be simplified to
\begin{multline*}
\left\Vert DS_{t}(u_{1})-DS_{t}(u_{2})\right\Vert \leq\\
Mc_{1}\intop_{0}^{t}\left\Vert DS_{s}(u_{1})-DS_{s}(u_{2})\right\Vert ds+Mc_{2}c_{4}t\left\Vert u_{1}-u_{2}\right\Vert _{L_{2}}\,.
\end{multline*}
Applying Gronwall's lemma yields
\[
\left\Vert DS_{t}(u_{1})-DS_{t}(u_{2})\right\Vert \leq Mc_{2}c_{4}te^{Mc_{1}t}\left\Vert u_{1}-u_{2}\right\Vert _{L_{2}}\,,
\]
which shows that the map $u\rightarrow DS_{t}(u)$ is continuous for
all $t\in[0,\tau]$.
\end{proof}
With Theorem \ref{thm:FrechetDerivative} and Theorem \ref{thm:ContinuousFrechet},
we can finally state our main result.
\begin{thm}
Consider a semilinear system (\ref{eq:Sys_Ax_f}) with a nonlinear
term $f$ that meets the conditions (B1) and (B2). If the linearized
system (\ref{eq:Linearized_system}) is exactly controllable on $[0,\tau]$,
then the original system is locally exactly controllable on $[0,\tau]$.
\end{thm}
\begin{proof}
We only have to show that the conditions (A1) and (A2) of Theorem
\ref{thm:ExactControllability} are satisfied. Condition (A1) follows
from Theorem \ref{thm:ContinuousFrechet}, and condition (A2) from
Theorem \ref{thm:FrechetDerivative} with $\bar{u}=0$.
\end{proof}

\section{Conclusion}

We have shown for a class of semilinear infinite-di\-men\-sion\-al
systems that exact controllability of the linearized system implies
local exact controllability of the original system. The assumptions
on the nonlinear operator are similar to those used in \cite{Smoller:1994}
for the linearized stability analysis, i.e., Lyapunov's indirect method.
Future research will deal with the question whether these assumptions
can be relaxed. A further interesting question is whether the approach
with the local surjectivity theorem can also be applied to the approximate
controllability problem. Such an extension is of course not at all
straightforward. Since the infinite-dimensional spaces in the local
surjectivity theorem are Banach spaces, i.e., complete, the intuitive
idea of simply using the reachable subspace of the linearized system
as target space is not directly applicable.